\documentclass[11pt]{amsart}
\usepackage{amssymb,amsbsy,upref,epsf,MnSymbol}
\usepackage[margin=1in]{geometry}

\newtheorem{theorem}{Theorem}
\newtheorem{proposition}[theorem]{Proposition}
\newtheorem{lemma}[theorem]{Lemma}

\newcommand{\R}{\mathbb{R}}
\newcommand{\E}{\mathcal{E}}
\newcommand{\M}{\mathcal{M}}
\newcommand{\varq}{\sigma}
\newcommand{\chevron}[1]{\langle #1 \rangle}
\newcommand{\norm}[1]{\left\lVert#1\right\rVert}
\newcommand{\paren}[1]{\left( #1 \right)}
\newcommand{\bracket}[1]{\left[ #1 \right]}
\newcommand{\abs}[1]{\left\lvert #1 \right\rvert}
\DeclareMathOperator{\supp}{supp}
\DeclareMathOperator{\trace}{tr}
\newcommand{\del}{\partial}
\newcommand{\grad}{\nabla}
\newcommand{\Laplace}{\Delta}
\newcommand{\ddt}{\frac{d}{dt}}
\renewcommand{\div}{\operatorname{div}}
\newcommand{\rest}{{\upharpoonright}}
\newcommand{\n}{^{-1}}
\newcommand{\indic}[1]{\chi_{\{#1\}}}
\newcounter{step_count}[section]
\newcommand{\step}[1]{\stepcounter{step_count} \smallskip \noindent{\bf Step \arabic{step_count}: #1}}

\title[De Giorgi techniques applied to Hamilton-Jacobi equations]{De Giorgi techniques applied to Hamilton-Jacobi equations with unbounded right-hand side}

\author[Stokols]{L. F. Stokols} 
\address[L. F. Stokols]{\newline Department of Mathematics, \newline The University of Texas at Austin, Austin, TX 78712, USA}
\email{lstokols@math.utexas.edu}

\author[Vasseur]{Alexis F. Vasseur}
\address[Alexis F. Vasseur]{\newline Department of Mathematics, \newline The University of Texas at Austin, Austin, TX 78712, USA}
\email{vasseur@math.utexas.edu}

\date{\today}


\thanks{\textbf{Acknowledgment.} A. F. Vasseur was partially supported by the NSF Grant DMS 1614918. }

\begin{document}

\begin{abstract}
In this article we obtain H\"{o}lder estimates for solutions to second-order Hamilton-Jacobi equations with super-quadratic growth in the gradient and unbounded source term.  The estimates are uniform with respect to the smallness of the diffusion and the smoothness of the Hamiltonian.  Our work is in the spirit of a result by P. Cardaliaguet and L. Silvestre \cite{cs}.  We utilize De Giorgi's method, which was introduced to this class of equations in \cite{cv}.  
\end{abstract}

\maketitle \centerline{\date}

\section{Introduction}
In the present paper, we study $C^\gamma$ regularization in solutions to a Hamilton-Jacobi evolution equation with viscosity:
\[ \del_t u + H(x,u,\grad u) - \varepsilon \Laplace u = 0, \qquad (t,x)\in (0,T)\times\Omega, \]
where $\Lambda > 0$, $\varepsilon \in [0,\Lambda]$, $\Omega \subseteq \R^n$, and the Hamiltonian has superquadratic growth in the gradient variable, uniform in $x$ and $t$:
\[ \frac{1}{\Lambda} \abs{v}^p - f(x,t) \leq H(t,x,z,v) \leq \Lambda \abs{v}^p + \Lambda, \qquad p>2, f \in L^m, m > 1+ \frac{\max(n,2)}{p}.\]

We will show that solutions are uniformly H\"{o}lder continuous away from the boundary of $\Omega$ and after a positive time has elapsed.  

Because $p>2$, it is the first order term that will dominate at small scales.  The second order term acts merely as a perturbation.  In fact, although our motivation is a first-order Hamilton-Jacobi equation with viscosity, our techniques can handle much more general second order terms.  Specifically, we will show the following theorem.  
\begin{theorem}[Main Theorem] \label{th:main}
Let constants $\Lambda > 0$, $\Lambda_0 \geq 0$, $p>2$, $m > 1+\frac{\max(n,2)}{p}$ be given, and let $\Omega \subseteq \R^n$ open and $T>0$ be given, and let $f \in L^m([0,T]\times \Omega)$ with $\norm{f}_m \leq \Lambda$ and a matrix $A \in L^\infty([0,T]\times\Omega;\R^{n\times n})$ with $\norm{A}_\infty \leq \Lambda$ be given, and let $\bar{\Omega} \subset \Omega$ compact and $0 < s < T$ be given.  

%Let
%\[ \ell > \sup\paren{ \paren{1 - \frac{2}{p}}\n + 1, \paren{1 - \frac{2}{p}}\n\paren{1-\frac{1}{m}}\n }. \]
There exists $0<\gamma<1$, depending on $p$, $\Lambda$, $\Lambda_0$, $m$, and $n$, such that any $u \in L^\infty((0,T)\times\Omega)$, $\grad u \in L^p$, satisfying 
\begin{equation} \label{eq:mainPDEdist}
\del_t u + \Lambda\n |\grad u|^p - \div (A \grad u) \leq f
\end{equation} 
in the sense of distributions, and satisfying 
\begin{equation} \label{eq:basePDEvisc}
\del_t u + \Lambda |\grad u|^p - \Lambda_0 m^-(D^2 u) \geq -\Lambda
\end{equation}
in the sense of viscosity, will have
\[ u \in C^\gamma((s,T)\times\bar{\Omega})\]
with norm depending on $\norm{u}_\infty$, $p$, $\Lambda$, $\Lambda_0$, $m$, $n$, $s$, and the distance between $\bar{\Omega}$ and $\R^n \setminus \Omega$.  
\end{theorem}

Here $m^-$ is a function that returns the lowest eigenvalue of a symmetric matrix, or 0 if all of the eigenvalues are positive.  For a function to solve Inequality~\eqref{eq:basePDEvisc} in the sense of viscosity means, following the definition of Barles~\cite{b13}, that the lower-semicontinuous envelope of that function is a viscosity supersolution of 
\[ \del_t u + \Lambda \abs{\grad u}^p - \Lambda_0 m^-(D^2 u) = -\Lambda. \]

Hamilton-Jacobi equations of this general form, with a viscosity term and polynomial growth in the gradient, were studied by Lasry and Lions \cite{ll89} in 1989, in connection with stochastic control problems.  For the case $p < 2$, this first-order-term can be viewed as a perturbation of a simple heat equation, and indeed solutions will be regular so long as the viscosity term is uniformly parabolic.  However, in the superquadratic case $p > 2$, it is the first order term which dominates at small scales, so standard parabolic theory does not apply.  

Schwab \cite{s} studied homogenization problems for Hamilton-Jacobi equations with superquadratic growth, which required him to prove that the regularity of solutions to these equations is independent of the regularity of the Hamiltonian.  His result still required, however, that the Hamiltonian be convex in $Du$.  It was Barles \cite{b} and Dolcetta, Leoni, and Porretta \cite{dlp} who noticed that convexity was unnecessary in the time-independent case, and Cardaliaguet (\cite{c}, \cite{cc}, \cite{cs}) for the time-dependent case.  

In the case that $f$ is bounded, Cardaliaguet and Silvestre (\cite{cs}, Theorem 1.2) showed H\"{o}lder continuity, using a second order term $m^+(D^2 u)$ instead of $\div(A \grad u)$ in \eqref{eq:mainPDEdist}.  In the case that $f$ is not assumed bounded, they could only show H\"{o}lder regularity with second order term $\trace(A D^2 u)$, $A \in C^1$.  Our result requires no regularity on $A$, at the expense of requiring that $\grad u \in L^p$ and $u$ solve Inequality~\eqref{eq:mainPDEdist} in the sense of distribution.  The motivation for considering $f$ unbounded is from Lasry and Lions \cite{ll07}.  

Most of the aforementioned results are proven by constructing super- and subsolutions.  In \cite{cv}, H\"{o}lder estimates are obtained, with $f$ bounded and no second order term, using a variation of De Giorgi's method.  The present work is a continuation of that project.  

The proof will proceed mostly along the same lines as De Giorgi \cite{dg} and \cite{cv}.  In the classical De Giorgi proof, in order to prove H\"{o}lder continuity one merely shows that if the function $u$ is ``mostly negative'' in some range of time, then the upper bound is improved in a later range of time.  If, alternatively, the function is not ``mostly negative,'' it must be ``mostly positive'' and hence one can apply the original argument to $-u$, improving the lower bound on $u$ in the same later range of time.  Either way, the $L^\infty$-bound of $u$ is improved in the later time range.  

In the sequel, the function $-u$ does not satisfy the same Inequality~\eqref{eq:mainPDEdist} as $u$.  However, time-reversed $-u$ does satisfy Inequality~\eqref{eq:mainPDEdist} with $A$ replaced by $-A$, since time reversal creates an extra minus sign on the $\del_t$ term. Thus unlike the classical De Giorgi proof, while the upper bound is improved in a later time range, the lower bound on $u$ is improved in an earlier time range, because time was reversed.  Note that while replacing $A$ by $-A$ should ostensibly cause great difficulty, the second order term is here a perturbation, and the first order term is the driver of regularization, so we can handle negative viscosities so long as the solution is known to exist and to be bounded.  

Next we must use the comparison principle in a small but crucial argument.  Based on Inequality~\eqref{eq:basePDEvisc}, a subsolution is constructed to show that a lower bound improvement in the early time range implies a smaller-but-still-positive improvement in the later time range.  This is referred to as ``flowing the improvement forward in time.''  

The key ingredient in improving the upper bound is an energy inequality.  Because of the second order term, we must multiply \eqref{eq:mainPDEdist} by $u_+$ to obtain the energy inequality (then we integrate by parts, and turn the second order term into a $|\grad u|^2$ term).  But the viscosity is a perturbation, and the true driver of the proof is the first order term.  Multiplying the first order term by $u_+$ yields $u_+ |\grad u_+|^p$, which is difficult because $u_+$ acts like a coefficient which is not bounded below.  Luckily, our goal is to bound $u$, and the difficulties only occur when $u_+$ is small.    

Section~\ref{sec:energy} derives an energy inequality, which quantifies the ellipticity of our equation.  Sections \ref{sec:DG1} and \ref{sec:DG2}  use the energy inequalities to prove De Giorgi's two lemmas.  Section~\ref{sec:flowforward} demonstrates how to flow the improvement forward in time, correcting for the necessary time reversal.  Finally, in Section~\ref{sec:main} we combine these lemmas to prove H\"{o}lder continuity.  A reader unfamiliar with De Giorgi-style proofs might want to begin with Section~\ref{sec:main}, lest the former sections seem unmotivated.  

Instead of proving continuity directly for $u$, it is preferable to consider
\[ \bar{u} := u + \Lambda t, \qquad \bar{f} := f + \Lambda\]
which satisfies the inequality
\begin{align}
%\del_t \bar{u} + \Lambda\n |\grad \bar{u}|^p - \div(A\grad \bar{u}) &\leq \bar{f}, 
\del_t \bar{u} + \Lambda |\grad \bar{u}|^p - \Lambda_0 m^-(D^2\bar{u}) &\geq 0. \label{eq:mainPDEvisc}
\end{align}

Note also that, by scaling our solution appropriately, we can assume that $\Lambda_0$ is arbitrarily small.  

Throughout this article, $C$ will indicate a constant which varies from line to line.  No two instances of the symbol should be assumed related to each other.

%****************************************************************************************
%****************************************************************************************
%****************************************************************************************
\section{The Energy Inequalities} \label{sec:energy}

\newcommand{\upl}{u_*}

We begin by deriving the Energy Inequalities, which play an analogous role to the Cacciopoli inequality in De Giorgi's original paper.  These inequalities serve to quantify the coercivity of the PDE in question.  We actually consider an infinite family of Energy Inequalities, corresponding to different entropies, indexed by the parameter $b$.  These inequalities must be valid even for non-positive matrices $A$.  

The lemma below claims three different forms for the Energy Inequality.  The first form will be used to compare distinct truncations of a solution in Section~\ref{sec:DG1}.  The second and third forms are only valid for large values of $b$, the former being used in Section~\ref{sec:DG1} and the latter being used in Section~\ref{sec:DG2}.  Notice that the gradient of $u$ appears in the right hand side of the first form, but not of the second or third forms.  

\begin{lemma}[Energy Inequality] \label{th:Energy} \label{th:alt_Energy}
Given $u$ verifying Inequality~\eqref{eq:mainPDEdist}, with $\norm{A}_{\infty},\norm{f}_m \leq \Lambda$, on some domain $[S,0]\times\Omega$, given constants $b$, $c$ and $S < T < 0$, and given $\phi$ a smooth non-negative function constant in time and compactly supported in $\Omega$, and defining $\upl = (u-c)_+$, then $\upl$ satisfies the inequality
%
%\begin{equation} \begin{gathered} \label{energy_1}
%\sup_{t \in [T,0]} \int \phi^2 \upl^{b+1}(t) + \iint_T^0 \phi^2 \upl^b \abs{\grad \upl}^p \\
%\leq K_{S,T,\phi} \paren{ \iint_S^0 (\upl^{b+1} + \upl^{b-1}|\grad \upl|^2) \indic{\phi} + \paren{\iint_S^0 \upl^{b m^\ast} \indic{\phi}}^\frac{1}{m^\ast} }. 
%\end{gathered} \end{equation}
\begin{equation} \begin{gathered} \label{energy_1}
\sup_{t \in [T,0]} \int \phi^2 \upl^{b+1}(t) + \iint_T^0 \phi^2 \upl^b \abs{\grad \upl}^p \\
\leq C(\Lambda, b) \paren{1+\frac{1}{T-S}} \paren{\norm{\phi}_\infty^2 + \norm{\grad \phi}_\infty^2} \bracket{ \iint_S^0 (\upl^{b+1} + \upl^{b-1}|\grad \upl|^2) \indic{\phi} + \paren{\iint_S^0 \upl^{b m^\ast} \indic{\phi}}^\frac{1}{m^\ast} }. 
\end{gathered} \end{equation}

Moreover, if $b > \sigma := \paren{1-\frac{2}{p}}\n$, then
\begin{equation} \begin{gathered} \label{energy_2}
\sup_{t \in [T,0]} \int \phi^2 \upl^{b+1}(t) + \iint_T^0 \phi^2 \upl^b \abs{\grad \upl}^p \\
\leq C(\Lambda, b) \paren{1+\frac{1}{T-S}} \paren{\norm{\phi}_\infty^2 + \norm{\grad \phi}_\infty^2} \bracket{ \iint_S^0 (\upl^{b+1} + \upl^{b-\sigma}) \indic{\phi} + \paren{\iint_S^0 \upl^{b m^\ast} \indic{\phi}}^\frac{1}{m^\ast} }.
\end{gathered} \end{equation}

If $b > \sigma$ but $\phi$ is not necessarily constant in time, then still we have
\begin{equation} \begin{gathered} \label{energy_3}
\chevron{\del_t(\upl^{b+1}),\phi^2}_{[S,0]\times\Omega} + \iint_{S}^0 \phi^2 \upl^b \abs{\grad \upl}^p \\
\leq C(\Lambda,b) \paren{ \iint_{S}^0 \phi^2 \upl^b f + \iint_{S}^0 \upl^{b+1} |\grad \phi|^2 + \iint_{S}^0 \phi^2 \upl^{b-\sigma} }. 
\end{gathered} \end{equation}

The integrals without limits are over all of $\Omega$, $\indic{\phi}$ means the indicator function of the support of $\phi$, and $m^\ast$ means the H\"{o}lder conjugate of $m$.
% and 
%\begin{equation} \label{eq:energy_constant} K_{S,T,\phi} = \paren{1+\frac{1}{T-S}} \paren{\norm{\phi}_\infty^2 + \norm{\grad \phi}_\infty^2} C(\Lambda). \end{equation}  
%the constant $C$ depends only on $\Lambda$.  
\end{lemma}

\begin{proof}
Formally, we want to integrate Inequality~\eqref{eq:mainPDEdist} against the test function $\phi^2 \upl^b$.  Because our solution $u$ is by assumption in $L^p(W^{1,p})$, the distributions $|\grad u|^p$ and $\div(A \grad u)$ both have enough regularity for this integration to make sense.  To justify our calculations on $\del_t u$, one can simply use the test function $\tau \ast (\phi^2 (\tau \ast \upl)^b)$ for $\tau$ some approximation to the identity and $\ast$ meaning convolution in time and space, though for reasons of clarity we drop the mollifiers in the formal calculations below.  

Multiply Inequality~\eqref{eq:mainPDEdist} by $\phi^2\upl^b$, then integrate over all of space $\Omega$:
\[\int \phi^2 \upl^b \del_t u + \Lambda\n \int \phi^2 \upl^b \abs{\grad u}^p + \int (\grad (\phi^2 \upl^b))A (\grad u) \leq \int \phi^2 \upl^b f. \] 
Notice that $D\upl = \indic{\upl > 0} Du$ for any first order differential operator $D$, so in the above expression we may replace every instance of $u$ with $\upl$.  By the product rule, $(b+1) \upl^b \del_t \upl = \del_t(\upl^{b+1})$.  Also, we can use the product rule and Young's Inequality to bound the $A$-term:
\begin{align*}% \label{A_calc}
\grad\paren{\phi^2 \upl^b} A \grad \upl &= b \phi^2 \upl^{b-1} (\grad\upl A \grad\upl) + 2 \phi \upl^b (\grad \upl A \grad \phi)
\\ &\leq b \Lambda \phi^2 \upl^{b-1} |\grad \upl|^2 + 2 \Lambda \paren{ \phi \upl^\frac{b-1}{2}|\grad \upl|} \paren{ \upl^\frac{b+1}{2} |\grad \phi|} 
\\ &\leq  b \Lambda \phi^2 \upl^{b-1} |\grad \upl|^2 + \Lambda \paren{ \phi \upl^\frac{b-1}{2}\grad \upl}^2 + \Lambda \paren{ \upl^\frac{b+1}{2} \grad \phi}^2
\\ &= (b+1) \Lambda \phi^2 \upl^{b-1} |\grad \upl|^2 + \Lambda \upl^{b+1} |\grad \phi|^2.
\end{align*}
Putting all of these together, we arrive at
\[ \frac{1}{b+1} \int \phi^2 \del_t(\upl^{b+1}) + \Lambda\n \int \phi^2 \upl^b \abs{\grad \upl}^p \leq \int \phi^2 \upl^b f + \Lambda \int \upl^{b+1} |\grad \phi|^2 + (b+1)\Lambda \int \phi^2 \upl^{b-1} |\grad \upl|^2. \]

If $b > \sigma$, then using Young's Inequality with exponents $p/2$ and $\sigma$, and a small constant $\eta$, we can break up the final term of the above inequality:
\begin{align*} 
\upl^{b-1}|\grad \upl|^2 &\leq C(p) \paren{ \paren{ \eta \upl^\frac{2b}{p} |\grad \upl|^2}^{p/2} + \paren{\frac{1}{\eta} \upl^{b\paren{1-\frac{2}{p}}-1}}^\sigma }
\\ &\leq C(p) \paren{ \eta^\frac{p}{2} \upl^b |\grad \upl|^p + \eta^{-\varq} \upl^{b-\varq} }. 
\end{align*}
By taking $\eta$ sufficiently small (depending on $p$, $b$, $\Lambda$), the $\upl^b|\grad \upl|^p$ term on the right can be absorbed by the same term with larger constant on the left.  We use the shorthand 
%\[ T(u,b) := \upl^{b-1} |\grad \upl|^2, \textrm{or} \upl^{b-\sigma}.\]
\[ T(\upl,b) := \begin{cases} 
      \upl^{b-1} |\grad \upl|^2 & \textrm{ if $b \leq \sigma$} \\
      \upl^{b-\sigma} & \textrm{ if $b > \sigma$} \\
   \end{cases} \]
and write
\[ \int \phi^2 \del_t(\upl^{b+1}) + \int \phi^2 \upl^b \abs{\grad \upl}^p \leq C(\Lambda,b) \paren{ \int \phi^2 \upl^b f + \int \upl^{b+1} |\grad \phi|^2 + \int \phi^2 T(\upl,b) }. \]

In the case that $\phi$ is time dependent, we can integrate the above in time to obtain \eqref{energy_3}.  From now on, we assume that $\del_t \phi = 0$, and hence $\int \phi^2 \del_t(\upl^{b+1}) = \ddt \int \phi^2 \upl^{b+1}$.  

For any times $s,t$ satisfying $S \leq s \leq T \leq t \leq 0$, we can integrate the above inequality over $[s,t]$ (and apply H\"{o}lder's to remove  dependence on $f$): 
\begin{gather*} 
\int \phi^2 \upl^{b+1}(t) + \iint_s^t \phi^2 \upl^b \abs{\grad \upl}^p \\
\leq C(\Lambda,b) \paren{ \int \phi^2 \upl^{b+1}(s) + \paren{ \iint_s^t (\phi^2 \upl^b)^{m^\ast} }^\frac{1}{m^\ast} + \iint_s^t \upl^{b+1} \abs{\grad \phi}^2 + \iint_s^t \phi^2 T(\upl,b)}. 
\end{gather*}
Due to our choice of $s,t$, the above inequality implies that
\begin{gather*}
 \int \phi^2 \upl^{b+1}(t) + \iint_T^t \phi^2 \upl^b \abs{\grad \upl}^p \\
 \leq C(\Lambda,b) \paren{ \int \phi^2 \upl^{b+1}(s) + \paren{ \iint_S^0 (\phi^2 \upl^b)^{m^\ast} }^\frac{1}{m^\ast} + \iint_S^0 \upl^{b+1} \abs{\grad \phi}^2 + \iint_S^0 \phi^2 T(\upl,b) }. 
\end{gather*}
Since the right side is independent of $t$, we can take a supremum of the left side over $T \leq t \leq 0$.  Add to this the inequality with $t=0$ to obtain
\begin{gather*}
\sup_{t\in[T,0]} \int \phi^2 \upl^{b+1}(t) + \iint_T^0 \phi^2 \upl^b \abs{\grad \upl}^p \\
\leq C(\Lambda,b) \paren{ \int \phi^2 \upl^{b+1}(s) + \paren{ \iint_S^0 (\phi^2 \upl^b)^{m^\ast} }^\frac{1}{m^\ast} + \iint_S^0 \upl^{b+1} \abs{\grad \phi}^2 + \iint_S^0 \phi^2 T(\upl,b) }.
\end{gather*}
Lastly, since this inequality holds for all $S \leq s \leq T$, it also holds if we average the right hand side over all values of $s$ in that range,
\begin{gather*}
\sup_{t \in [T,0]} \int \phi^2 \upl^{b+1}(t) + \iint_T^0 \phi^2 \upl^b \abs{\grad \upl}^p \\
\leq  C(\Lambda,b) \paren{ \frac{1}{T\!\! - \!\! S} \iint_S^T \phi^2 \upl^{b+1} + \paren{ \iint_S^0 (\phi^2 \upl^b)^{m^\ast} }^\frac{1}{m^\ast} + \iint_S^0 \upl^{b+1} \abs{\grad \phi}^2 + \iint_S^0 \phi^2 T(\upl,b) }.
\end{gather*}

From here the result follows naturally.  
\end{proof}

%****************************************************************************************
%****************************************************************************************
%****************************************************************************************
\section{De Giorgi's first lemma}\label{sec:DG1}

Now we present De Giorgi's first lemma.  
If we define
\[ \overline{Q}_2 := [-2,0]\times B_2, \qquad Q_1 := [-1,0]\times B_1,\]  
this lemma tells us that the supremum in $Q_1$ of solutions to \eqref{eq:mainPDEdist} can be controlled by the measure of $\{u > 0\}$ in $\overline{Q}_2$.  

\begin{proposition}[De Giorgi's First Lemma] \label{th:DG1}
There exists a constant $\delta_0 > 0$ depending only on $\Lambda$, $p$, $m$, and the dimension such that, for any $u$ satisfying Inequality~\eqref{eq:mainPDEdist} on $\overline{Q}_2$ in the sense of distributions, the following implication holds:

If 
\[ u(t,x) \leq 1 \qquad \forall\, (t,x) \in \overline{Q}_2 \]
and
\[ \big|\{u > 0\} \cap \overline{Q}_2 \big| \leq \delta_0,\] 
%
%If 
%\[ \norm{u_+}_{L^{\ell}(\overline{Q}_2)}^{\ell} \leq \delta_0 \]
then 
\[ u(t,x) \leq \frac{1}{2} \qquad \forall \, (t,x) \in Q_1.\]

%Note $u_+ = \max(u,0)$ and $\sigma = \paren{1-\frac{2}{p}}\n$.  

\end{proposition} 

De Giorgi's first lemma is proved by cutting off $u$ at larger and larger values, and showing that as the cutoff value tends to $1/2$, some Lebesgue norm of the remainder tends to zero.  

\begin{proof}
Let us specify the sequence of cutoffs.  We'll consider 
\begin{itemize}
\item heights $C_k = \frac{1}{2}-2^{-k-1}$ from $C_0 = 0$ to $C_\infty = \frac{1}{2}$ with $C_k - C_{k-1} = 2^{-k-1}$; 
\item functions $u_k = \max(u-C_k,0)$ from $u_0 = u_+$ to $u_\infty = (u-\frac{1}{2})_+$;
\item balls $B^k$ of radius $1 + 2^{-k}$ from $B^0 = B_2 = \{x: |x|<2\}$ to $B^\infty = B_1 = \{x:|x|<1\}$; 
\item times $T_k = -1-2^{-k}$ from $T_0 = -2$ to $T_\infty = -1$ with $T_k - T_{k-1} = 2^{-k}$; 
\item and smooth functions $\phi_k$ such that $\supp(\phi_k) = B^k$ and $\phi_k \rest B^{k+1} \equiv 1$, with $0 \leq \phi_k \leq 1$ and $\abs{\grad \phi_k} \leq 2^{k+2}$.  
\end{itemize}

Define the "energy" of the $k^\textrm{th}$ level to be
\[ \E_k := \sup_{t \in [T_{k+1},0]} \int (\phi_k u_k)^2(t) + \iint_{k+1} \phi_k^2 u_k \abs{\grad u_k}^p,\]  
where $\iint_k$ means $\int_{T_k}^0 \int_{\R^n}$.  First we will show via Sobolev's inequality that this energy term controls some $L^{(1+\beta)q}$ norm of $\phi_k u_k$.  Then we will show via the Energy Inequality that the same $L^{(1+\beta)q}$ norm controls this energy term.  

\step{Controlling $L^{(1+\beta)q}$ using $\E_k$}

Before we can apply Sobolev's inequality, we have to deal with the inhomogeneity of the gradient term.  We do this by "going up a level" from $u_k$ to $u_{k+1}$.  
\begin{align*}
\E_k &\geq \iint_{k+1} \phi_k^2 u_k \abs{\grad u_k}^p 
\\ &\geq \iint_{k+1} \phi_k^2 \left[ 2^{-(k+2)} \indic{u_k \geq 2^{-k-2}} \right] \abs{\grad u_k}^p
\\ &= 2^{-k-2}\iint_{k+1} \phi_k^2 \indic{u_{k+1} \geq 0} \abs{\grad u_k}^p
\\ &= 2^{-k-2}\iint_{k+1} \phi_k^2 \abs{\grad u_{k+1}}^p
\\&\geq 2^{-k-2}\iint_{k+1} \indic{B^{k+1}} \abs{\grad u_{k+1}}^p
\\&= C^{-k} \int_{T_{k+1}}^0 \norm{\grad u_{k+1}}_{L^p(B^{k+1})}^p
\\&= C^{-k} \norm{\grad u_{k+1}}_{L^p([T_{k+1},0];L^p(B^{k+1}))}^p
\end{align*}

We introduce now a parameter $\beta \in (0,1]$, satisfying
\begin{equation*} 0 < \frac{1}{n} - \frac{\beta}{2} < \frac{1}{p}, \qquad n \geq 2 \end{equation*}
or $\beta=1$ if $n=1$.  We are going to apply Sobolev's Inequality to bound the $L^{p'}$ norm of $u_k^{1+\beta}$ by some Lebesgue norm of $\grad u_k^{1+\beta}$.  

%We have bounds, in terms of $\E_k$, on the $L^\infty(L^2)$ norm of $u_{k+1}$ as well as the $L^p(L^p)$ norm of $\grad u_{k+1}$.  
%Our next step would be to apply the Sobolev Embedding theorem, but that theorem only applies
%If $n > p$, we could apply Sobolev's inequality immediately, but 
%We want to turn this bound on the norm of $\grad u_{k+1}$ into a bound on the norm of $\grad u_{k+1}^2 = u_{k+1} \grad u_{k+1}^2$.  
Since 
\[ \norm{ u_{k+1}^\beta }_{L^\infty([T_{k+1},0];L^{2/\beta}(B^{k+1}))}^{2/\beta} = \sup_{t \in [T_{k+1},0]}  \norm{ u_{k+1}(t) }_{L^2(B^{k+1})}^2 \leq \sup_{t \in [T_{k+1},0]} \norm{\phi_k u_k(t)}_{L^2(B^{k+1})}^2 \leq \E_k, \]
we know by elementary properties of Lebesgue spaces that
\begin{align} \begin{split} \label{grad_1+beta_bdd}
\int_{T_{k+1}}^0 \norm{\grad u_{k+1}^{\beta+1}}_{L^\frac{2p}{2+p\beta}(B^{k+1})}^p &= \norm{ u_{k+1}^\beta \grad u_{k+1}}_{L^p([T_{k+1},0];L^\frac{2p}{2+p\beta}(B^{k+1}))}^p
\\ &\leq  \norm{ u_{k+1}^\beta }_{L^\infty([T_{k+1},0];L^{2/\beta}(B^{k+1}))}^p \norm{\grad u_{k+1}}_{L^p([T_{k+1},0];L^p(B^{k+1}))}^p
%\int_{T_{k+1}}^0 \norm{u_{k+1}^\beta \grad u_{k+1}}_{L^{\frac{2p}{2+p\beta}}(B^{k+1})}^p 
\\ &\leq \paren{\E_k^{\beta/2}}^p C^k \E_k  = C^k \E_k^{1 + \frac{p\beta}{2}}. 
\end{split} \end{align}

%By the Sobolev Embedding Theorem, the $L^\frac{2p}{2+p\beta}$ norm of $\grad u_{k+1}^{1+\beta}$ can be used to control the $L^{p'}$ norm of $u_{k+1}^{1+\beta}$, where $p'$ is some finite number greater than $\frac{2p}{2+p\beta}$.  
If $n > 1$, then let $\frac{1}{p'} = \frac{2+p\beta}{2p} - \frac{1}{n} = \frac{\beta}{2} + \frac{1}{p} - \frac{1}{n}$.  If $n=1$, then take $p' = p$ (which renders some of the following calculations trivial).  
Sobolev Embedding yields
\begin{align*}
\norm{u_{k+1}^{1+\beta}}_{L^{p'}\left(B^{k+1}\right)} &\leq \norm{u_{k+1}^{1+\beta} - \strokedint_{B^{k+1}} u_{k+1}^{1+\beta} }_{L^{p'}\left(B^{k+1}\right)} + |B^{k+1}|^{\frac{1}{p'}-1} \int_{B^{k+1}} u_{k+1}^{1+\beta}
\\ &\leq C \paren{ \norm{\grad u_{k+1}^{1+\beta}}_{L^\frac{2p}{2+p\beta}\left(B^{k+1}\right)} + \norm{u_{k+1}}_{L^2\left(B^{k+1}\right)}^{1+\beta} }.
\end{align*}
Remember that $\strokedint_E := \frac{1}{|E|}\int_E$, and $1+\beta \leq 2$ so $L^{1+\beta} \subseteq L^2$.  
%\[ \norm{\grad u_{k+1}}_{L^p\left(B^{k+1}\right)} \geq C \norm{u_{k+1} - \strokedint u_{k+1} }_{L^{p'}\left(B^{k+1}\right)} ,\] 

With the above calculation and \eqref{grad_1+beta_bdd}, we can estimate
%\[\E_k \geq C 2^{-k-2}\int_{T_{k+1}}^0 \paren{\int \abs{\phi_{k+1} u_{k+1}}^{p'}}^\frac{p}{p'}.\]
\begin{align*} 
\int_{T_{k+1}}^0 \norm{u_{k+1}^{1+\beta}}_{L^{p'}(B_{k+1})}^p &\leq C \int_{T_{k+1}}^0 \paren{\norm{\grad u_{k+1}^{1+\beta}}_{L^\frac{2p}{2+p\beta}(B^{k+1})} + \norm{u_{k+1}}_{L^2(B^{k+1})}^{1+\beta}}^p
\\ &\leq C \paren{ \int_{T_{k+1}}^0 \norm{\grad u_{k+1}^{1+\beta}}_{L^\frac{2p}{2+p\beta}(B^{k+1})}^p + T_{k+1} \sup_{t\in [T_{k+1},0]} \norm{u_{k+1}(t)}_{L^2(B^{k+1})}^{p(1+\beta)}}
\\ &\leq C \paren{C^k \E_k^{1+\frac{p\beta}{2}} + \E_k^{p\frac{1+\beta}{2}}}
\\ &\leq C^k \E_k^{1 + \frac{p\beta}{2}}.
\end{align*}
This last estimate holds as long as $\E_k$ is less than one.  

%One can see that $\E_k$ controls the $L^\infty(L^\frac{2}{1+\beta})$ norm of $u_k^{1+\beta}$ (and hence the smaller $u_{k+1}^{1+\beta}$), and we now know that it controls the $L^p(L^{p'})$ norm thereof, so we can use Riesz-Thorin interpolation.  Since $\infty > \frac{2}{1+\beta}$ but $p \leq p'$, there must be some number $q$ such that Riesz-Thorin would allow us to control the $L^q(L^q)$ norm.  Specifically,
We wish to apply the Riesz-Thorin theorem to interpolate between the $L^p(L^{p'})$ and $L^\infty(L^\frac{2}{1+\beta})$ norms of $u_{k+1}^{1+\beta}$.  First define
\begin{equation}\label{def_of_q} q = p + \paren{1 - \frac{p}{p'}}\frac{2}{1+\beta}. \end{equation}
Because $p' \geq p$ and hence $q \geq p$, we can let $\theta = \frac{p}{q} \in [0,1]$ and interpolate to obtain
\[ \paren{1-\theta} \frac{1}{\infty} + \theta \frac{1}{p} 
= 0 + \frac{1}{q} 
= \frac{1}{q} \]
and
\begin{align*} 
\paren{1-\theta} \frac{1+\beta}{2} + \theta \, \frac{1}{p'} 
&= \paren{\frac{q-p}{q}} \frac{1+\beta}{2} + \frac{1}{q} \paren{\frac{p}{p'}} 
\\ &= \frac{1}{q} \paren{1 -  \frac{p}{p'}} \paren{\frac{2}{1+\beta}} \frac{1+\beta}{2} + \paren{\frac{p}{p'}} \frac{1}{q}.
\\ &= \frac{1}{q}. 
\end{align*}

Therefore the Riesz-Thorin interpolation theorem yields
\begin{align*}
\norm{u_{k+1}^{1+\beta}}_{L^q\left([T_{k+1},0]\times B^{k+1}\right)} &\leq C \left[ \norm{ u_{k+1}^{1+\beta} }_{L^\infty\left([T_{k+1},0];L^\frac{2}{1+\beta}(B^{k+1})\right)} \right]^{1 - \theta} \left[ \norm{ u_{k+1}^{1+\beta} }_{L^p\left([T_{k+1},0];L^{p'}(B^{k+1})\right)} \right]^\theta
\\ &\leq C \left[ \sup_{t \in [T_{k+1},0]} \norm{\phi_k u_k}_{L^2(B^k)}^{1+\beta} \right]^{1-\frac{p}{q}} \left[ \paren{C^k\E_k^{1 + \frac{p\beta}{2}}}^{1/p} \right]^\frac{p}{q} 
\\ &\leq C^k \left[\E_k^{\frac{1}{2} + \frac{\beta}{2}} \right]^{1-\frac{p}{q}} \E_k^{\frac{1}{q} + \frac{\beta}{2} \cdot \frac{p}{q}}
\\ &= C^k \E_k^{\frac{1}{q} + \frac{1}{2}\paren{1+\beta-\frac{p}{q}}}. 
\end{align*}

Thus finally, 
\begin{equation} \label{q_leq_E}
\iint_{k+1} \abs{\phi_{k+1}u_{k+1}}^{(1+\beta)q} \leq \iint_{k+1} \indic{B^{k+1}} (u_{k+1}^{1+\beta})^q \leq C^k \E_k^{1 + \frac{(1+\beta)q-p}{2}}.
\end{equation}

\step{A Recursive relation for the sequence $\E_k$}

%The Energy Inequality, applied to $u_k$ with $b=1$, $\phi_k$, and times $T_{k+1}$ and $T_k$, calculating out the constant~\eqref{eq:energy_constant}, tells us that
%\begin{equation} \label{energy_for_dg1} \E_k \leq C 2^{k+2} \paren{ \iint_k (u_k^2 + |\grad u_k|^2) \indic{B^k} + \paren{ \iint_k u_k^{m^\ast} \indic{B^k} }^{1/m^\ast} }. \end{equation}
%Now that we have \eqref{q_leq_E}, we are ready to bound the three terms on this inequality's right hand side.  
%
%First we have to confirm that the exponents $2$ and $m^\ast$ are both less than $(1+\beta)q$.  Remember that
%\begin{equation} \label{one_plus_beta_q} (1+\beta)q = 2 + (1+\beta)p - \frac{2p}{p'}. \end{equation} % \qquad \textrm{and} \quad p' > p > 2, \]
%Since $p' \geq p > 2$, we see $2 < (1+\beta)q$.  If $n > 1$, then $\frac{1}{p'} = \frac{1}{p} + \frac{\beta}{2} - \frac{1}{n}$ so $2\frac{p}{p'} = 2 + p\beta - 2 \frac{p}{n}$ and 
%\begin{equation} \label{n_greater_1_version_of_one_plus_beta_q} (1+\beta)q = p + 2\frac{p}{n}, \end{equation}
% which is larger than $1 + \frac{p}{n} > m^\ast$ by the defining assumption on $m$.  If $n=1$, then $p' = p$ and $\beta = 1$, so $(1+\beta)q = (1+\beta)p \geq 1+ \frac{p}{n}$.  
%%and $p' \geq p > 2$, so $2 < (1+\beta)q$.  Moreovere, $m^\ast \leq 1 + \frac{p}{n} \leq 1 + p$

Recall from the definition~\eqref{def_of_q} of $q$ that $(1+\beta)q = 2 + (1+\beta)p - 2\frac{p}{p'}$.  If $n > 1$, then by the definition of $p'$ we have that $2\frac{p}{p'} = 2 + p\beta - 2 \frac{p}{n}$.  
% definition of $p'$ we have $2\frac{p}{p'} = 2 + p\beta - 2 \frac{p}{n}$, and also $\frac{1}{p} - \frac{1}{n} > -\frac{\beta}{2}$ by the definition of $\beta$.  
If $n=1$, then $p' = p$ and $\beta = 1$.
Therefore,
%$\frac{1}{p'} = \frac{1}{p} + \frac{\beta}{2} - \frac{1}{n}$ and $1/n - \beta/2 < 1/p$ by definition.  If $n=1$, recall that $p' = p$ and $\beta=1$.  One can calculate directly that
\begin{align} \begin{split} \label{useful_q_facts}
(1+\beta)q &= p + 2 \frac{p}{n}, \qquad n > 1 \\
(1+\beta)q &= 2p, \qquad \qquad n = 1.
\end{split} \end{align}

% q = p + \paren{1 - p/p'}\frac{2}{1+\beta}$

The Energy Inequality~\eqref{energy_1}, applied to $u_k$ with $b=1$, $\phi_k$, and times $T_{k+1}$ and $T_k$, tells us that
\begin{equation} \label{energy_for_dg1} \E_k \leq C 2^{k+2} \paren{ \iint_k (u_k^2 + |\grad u_k|^2) \indic{B^k} + \paren{ \iint_k u_k^{m^\ast} \indic{B^k} }^{1/m^\ast} }. \end{equation}
Now that we have \eqref{q_leq_E}, we are ready to bound the three terms on this inequality's right hand side.  

For the first and third terms on the right hand side, we can use a well known trick of De Giorgi \cite{dg}.  For any $j \leq (1+\beta)q$ we can apply H\"{o}lder's inequality followed by Chebyshev's inequality to obtain
\begin{align*}
\iint_k u_k^j \indic{B^k} &= \iint_k (\phi_{k-1} u_k)^j \indic{B^k \cap \{u_{k-1} > 2^{-(k+1)}\}}
\\ &\leq \paren{\iint_k (\phi_{k-1} u_k)^{(1+\beta)q}}^{j/[(1+\beta)q]} \abs{\{\phi_{k-1} u_{k-1} > 2^{-(k+1)}\}}^{1-j/[(1+\beta)q]}
\\ &\leq \paren{\iint_{k-1} (\phi_{k-1} u_{k-1})^{(1+\beta)q}}^{j/[(1+\beta)q]}\abs{\{(\phi_{k-1} u_{k-1})^{(1+\beta)q} > 2^{-(k+1)(1+\beta)q}\}}^{1-j/[(1+\beta)q]}
\\ &\leq \paren{\iint_{k-1} (\phi_{k-1} u_{k-1})^{(1+\beta)q}}^{j/[(1+\beta)q]}\paren{2^{(k+1)(1+\beta)q}\iint_{k-1} (\phi_{k-1}u_{k-1})^{(1+\beta)q}}^{1-j/[(1+\beta)q]}
\\ &\leq 2^{(k+1)((1+\beta)q-j)} \iint_{k-1} (\phi_{k-1}u_{k-1})^{(1+\beta)q}
\\ &\leq C^k \E_{k-2}^{1+\frac{(1+\beta)q-p}{2}}.
\end{align*}
We know from \eqref{useful_q_facts} that $2 < (1+\beta)q$ and $m\ast \leq 1 + \frac{p}{n} \leq (1+\beta)q$, so setting $j=2$ and $j = m^\ast$ gives us bounds on the first and third terms of \eqref{energy_for_dg1}, respectively.  

For the second term of \eqref{energy_for_dg1}, calculate
\begin{align*}
\iint_k |\grad u_k|^2 \indic{B^k} &\leq \iint_k \phi_{k-1}^{4/p} \indic{u_k > 0}|\grad u_{k-1}|^2 \indic{\phi_k u_k > 0}
\\ &\leq \paren{ \iint_k \phi_{k-1}^2 \indic{u_{k-1} > 2^{-(k+1)}} |\grad u_{k-1}|^p }^{2/p} \abs{\{\phi_{k-1} u_{k-1} > 2^{-(k+1)}\}}^{1-2/p}
\\ &\leq \paren{ 2^{k+1} \iint_k \phi_{k-1}^2 u_{k-1} |\grad u_{k-1}|^p }^{2/p} \abs{\{(\phi_{k-1} u_{k-1})^{(1+\beta)q} > 2^{-(k+1)(1+\beta)q}\}}^{1-2/p}
\\ &\leq \paren{ 2^{k+1} \E_{k-1} }^{2/p} \paren{ 2^{(k+1)(1+\beta)q} \iint_{k-1} (\phi_{k-1} u_{k-1})^{(1+\beta)q} }^{1-2/p}
\\ &\leq \paren{ 2^{k+1} \E_{k-2} }^{2/p} \paren{ 2^{(k+1)(1+\beta)q} C^{k-2} \E_{k-2}^{1+\frac{(1+\beta)q-p}{2}} }^{1-2/p}
\\ &\leq C^k \E_{k-2}^{1 + \paren{1-\frac{2}{p}}\frac{(1+\beta)q-p}{2}}.
\end{align*}
The second-to-last inequality used \eqref{q_leq_E}, and the fact that $\E_{k-1} \leq \E_{k-2}$.  

Finally we have the recursive relation 
\begin{equation}
\E_k \leq C^k \paren{\E_{k-2}^{1 + \frac{(1+\beta)q-p}{2}} + \E_{k-2}^{1 + \paren{1-\frac{2}{p}}\frac{(1+\beta)q-p}{2} } + \E_{k-2}^{\paren{1 + \frac{(1+\beta)q-p}{2}} \paren{\frac{1}{m^\ast}}} }. 
\end{equation}
From \eqref{useful_q_facts} and $p > 2$, one sees that the first two of these exponents are strictly greater than 1.  From \eqref{useful_q_facts} and $m^\ast < 1 + \frac{p}{n}$, one sees that the third exponent is strictly greater than 1.  

%By our restriction that $m > \max\paren{2,1+n/p}$, and that $q > p + \frac{2}{m-1}$, we know that each of these exponents is strictly greater than one.  

Because we can assume wlog that all $\E_k$ are small, this simplifies for our purposes to 
\[ \E_k \leq C^k  \E_{k-2}^{1 + \varepsilon }. \]

Therefore the sequence $\E_{2n+1}$ is bounded by a sequence $a_{n+1} = c^n a_n^{1+\varepsilon}$, $a_0 = \E_1$.  Because the exponent is greater than one, the bounding sequence will tend to zero as long as $a_0$ is sufficiently small.  

But since $u \leq 1$ by assumption, we can calculate, for any $b > \sigma$,
\begin{align*} 
\E_1 &= \sup_{[T_1,0]} \int \phi_1^2 u_1^2 + \iint_1 \phi_1^2 u_1 \abs{\grad u_1}^p
\\ &= 2^{2 (b-1)} \paren{ \sup_{[T_1,0]} \int \phi_1^2 u_1^2 \paren{ 2^{-2} \indic{u_0 > 2^{-2}} }^{b-1} + \iint_1 \phi_1^2 u_1 \paren{ 2^{-2} \indic{u_0 > 2^{-2}} }^{b-1} \abs{\grad u_1}^p }
\\ &\leq 2^{2 (b-1)} \paren{ \sup_{[T_1,0]} \int \phi_1^2 u_1^2 u_0^{b-1} + \iint_1 \phi_1^2 u_1 u_0^{b-1} \abs{\grad u_1}^p }
\\ &\leq 2^{2 (b-1)} \paren{ \sup_{[T_1,0]} \int \phi_0^2 u_0^{b+1} + \iint_1 \phi_0^2 u_0^b \abs{\grad u_1}^p }
\\ &\leq C \paren{ \iint_0 (u_0^{b+1} + u_0^{b-\sigma}) \indic{B^0} + \paren{\iint_0 u_0^{b m^\ast} \indic{B^0}}^\frac{1}{m^\ast} }
\\ &\leq C \paren{|\{u>0\}\cap \overline{Q}_2| + |\{u>0\}\cap \overline{Q}_2| + |\{u>0\}\cap \overline{Q}_2|^{1/m^\ast}}.
%\\ &\leq C \norm{u_+}_{L^\ell(\overline{Q}_2)}
\end{align*}
%by definition of $\ell$.  
Therefore there exists a $\delta_0>0$ sufficiently small that, if $\big|\{u > 0\} \cap \overline{Q}_2 \big| \leq \delta_0$, then $\E_1$ will be small enough that $\E_k \to 0$ as $k \to \infty$.  

%But we have
%\begin{align*}
%\E_1 &\leq C \iint_1 (u_1 + u_1^2 + |\grad u_1|^2) \indic{B^1}
%\\ &\leq C (\norm{u_1}_{\sigma+2} + \norm{u_1}_{\sigma+2}^2 + \norm{\grad u_1}_p^2)
%\\ &\leq C \paren{\norm{u_0}_{\sigma+2} + \norm{u_0}_{\sigma+2}^2 + \paren{2^2\iint_0 \phi_0^2 u_0^{\sigma+1} |\grad u_0|^p}^{2/p}},
%\end{align*}
%and by the Energy Inequality with $b = \sigma+1$, $\norm{u_0}_{\sigma+2}$ controls $\int\phi_0^2 u_0^{\sigma+1} |\grad u_0|^p$.  Therefore, there exists a $\delta_0>0$ so that, if $\norm{u_+}_{\sigma+2}^{\sigma+2} < \delta_0$, then $\E_1$ will be small enough that $\E_k \to 0$ as $k \to \infty$.   

If $\E_k \to 0$, then
\[ \norm{u_k}_{L^q([-1,0]\times B_1)} \leq \norm{\phi_k u_k}_{L^q([T_k,0]\times B^k)} \leq C^k \E_k^{\frac{1}{q} + \frac{q-p}{2q}} \to 0.\]  
By the monotone convergence theorem, we conclude that $\norm{(u-1/2)_+}_{L^q([-1,0]\times B_1)} = 0$ and so
\[ |\{ u > \frac{1}{2} \} \cap [-1,0]\times B_1| = 0.\]  

\end{proof}

%****************************************************************************************
%****************************************************************************************
%****************************************************************************************
\section{De Giorgi's second lemma} \label{sec:DG2}

The second De Giorgi lemma is a quantitative version of the statement ``solutions to our PDE cannot have jump discontinuities.''  

Define the sets
\[ Q_3 = [-4,0]\times B_3, \qquad Q_2 = [-4,0]\times B_2, \]
and remember that 
\[ \qquad \overline{Q}_2 = [-2,0]\times B_2.\]  
According to the next theorem, if a solution to \eqref{eq:mainPDEdist} is negative in $Q_2$ on a set of large measure, and $\geq 1$ in $\overline{Q}_2$ on a set of large measure, and it is bounded on all of $Q_3$, then that solution must be strictly between 0 and 1 on a set of large measure in $\overline{Q}_2$.  

The proof is by compactness.  Because the solution is bounded on $Q_3$, we can use the Energy nequality to bound its derivatives on $Q_2$.  By a theorem of Aubin and Lions, which is an instance of the general principle ``bounded derivatives imply compactness,'' we can conclude that the family of bounded solutions is precompact.  Therefore, if the interstitial measure is not bounded below, there must be a limit function which would have both bounded derivatives and a jump discontinuity, a contradiction.  

Because of the coefficient on $|\grad u|$ in the Energy Inequality, the derivatives are not well controlled when $u$ is near zero.  This is solved by considering instead $u$ raised to some power, whose derivatives are trivially controlled when $u$ is near zero, and whose convergence implies the convergence of $u$.  

\begin{proposition}[De Giorgi's Second Lemma]\label{th:DG2}
There exists a positive constant $\mu_0$ depending on $\Lambda$, $p$, $m$, $\delta_0$, and the dimension, such that for any $u$
satisfying Inequality~\eqref{eq:mainPDEdist} in the sense of distributions, with 
\[ u(t,x) \leq 2 \qquad \forall\,(t,x) \in Q_3\]
and
\[\abs{\{u \leq 0\} \cap Q_2} \geq \frac{\abs{Q_2}}{2},\] 
and, for $\delta_0$ the quantity divined in Proposition~\ref{th:DG1},
\[ \abs{\{u \geq 1\} \cap \overline{Q}_2} \geq \delta_0,\]  
it must be the case that
\[ \abs{\{0 < u < 1\} \cap Q_2} \geq \mu_0.\]  
\end{proposition}

\begin{proof}
Suppose the proposition is false.  Then we can consider a sequence $u_i$ of functions which satisfy all the hypotheses of this proposition but for which 
\[\abs{\{0 < u_i < 1\} \cap Q_2} \leq \frac{1}{i}.\]  
%Note that because we want a bound which is independent of $A$, each $u_i$ might satisfy a version of \eqref{eq:mainPDEdist} with a distinct matrix $A_i$.  

Rather than seek a limit of the sequence $u_i$, we will actually seek a limit of $(u_i)_+^{\sigma+2}$, where $\frac{1}{\sigma} + \frac{2}{p} = 1$ consistent with the notation in Lemma~\ref{th:Energy}.  First we need to bound the space and time derivatives of $(u_i)_+^{\sigma+2}$ uniformly in $i$.  

\step{Bounding the derivatives}

To bound the spatial derivatives, we use the Energy Inequality~\eqref{energy_2} with $b = (\sigma+1)p$, and choose a smooth cutoff function $\phi$ satisfying
\[ \phi:B_3 \to [0,1], \qquad \phi \geq 0, \qquad \supp(\phi) \textrm{ compact}, \qquad \psi(x)=1 \:\; \forall \, x \in B_2. \]

By the Energy Inequality, we have
\begin{align*} 
\iint_{B_2\times [-4,0]} |\grad (u_i)_+^{\sigma +2}|^p &\leq (\sigma +2)^p \iint_{-4}^0 \psi (u_i)_+^{p(\sigma +1)} |\grad (u_i)_+|^p 
\\ &\leq C \iint_{-4}^0 \paren{ (u_i)_+^{p(\sigma+1)-\varq} + (u_i)_+^{p(\sigma+1)+1} } \indic{B_3} + C \paren{\iint_{-4}^0 (u_i)_+^{m^\ast p(\sigma +1)} \indic{B_3} }^{1/m^\ast}
\\ &\leq C(\Lambda,p,n,m).
\end{align*}
Therefore the sequence $\grad (u_i)_+^{\sigma+2}$ is bounded in $L^p([-4,0];L^p(B_2))$ uniformly in $i$.

Bounding the time derivative is much more involved.  We will show that $\del_t (u_i)_+^{\sigma+2}$ are uniformly bounded in $\M([-4,0];W^{-1,\infty})$, where $\M$ means the dual space to $L^\infty$ and $W^{-1,\infty}$ is the dual of $\overline{C^\infty_0(B_2) \cap W^{1,\infty}(B_2)}$.  

%The Energy Inequality~\eqref{energy_3} gives us a nice bound on the positive part of $\del_t(u_i)_+^{\sigma+2}$, but in an inconvenient form.  We want to bound this quantity in a distributional sense, but the Energy Inequality requires us to use a test function which is a perfect square.  Specifically, 
Using the Energy Inequality~\eqref{energy_3} with $b = \sigma+1$ and any test function $\varphi: Q_3 \to \R$ which is smooth and compactly supported in space (but not necessarily compactly supported in time), together with the fact that $\norm{f}_1 \leq \norm{f}_m \leq \Lambda$ and $u_i \leq 2$, gives us the bound
\begin{align*}
\chevron{\del_t (u_i)_+^{\sigma+2}, \varphi^2}_{[-4,0]\times B_3} &\leq C(p,\Lambda) \paren{ \iint \varphi^2 (u_i)_+^{\sigma+1}f + \iint \varphi^2 (u_i)_+ + \iint (u_i)_+^{\sigma+2} |\grad \varphi|^2 }
\\ &\leq C(p,\Lambda) \paren{ \norm{\varphi}_{L^\infty(Q_3)}^2 + \norm{\grad \varphi}_{L^\infty(Q_3)}^2 }.
\end{align*}

% $\psi \in L^\infty([-4,0],W^{1\infty}(B_2))$.

We must find a similar bound on $\chevron{\del_t (u_i)_+^{\sigma+1}, \psi}$ when $\psi$ is not necessarily the square of a smooth function.  Our strategy is to decompose $\psi$ as a sum of a perfect square and a function independent of time.  To this end, define $\sqrt{\phi}$ a specific smooth function (of space only) supported in $B_3$ and identically 1 on $B_2$.  Then $\phi := \sqrt{\phi}^2$ will also be smooth, supported on $B_3$, and identically 1 on $B_2$.  

Consider any $\psi \in C_0^\infty(Q_3)$, and set $K = \norm{\psi}_\infty + \norm{\grad \psi}_\infty$.  Here and in the sequel, $\norm{\boldsymbol{\cdot}}_\infty$ means $\norm{\boldsymbol{\cdot}}_{L^\infty(Q_3)}$.  Note that $\psi + K \phi$ is non-negative, so we can define $\varphi$ by the relation
\[ \psi = \varphi^2 - K \phi.\]  
Estimate
\begin{align*}
\iint_{Q_2} \psi \del_t (u_i)_+^{\sigma+2} &= -K\iint_{Q_3}\phi \del_t (u_i)_+^{\sigma+2} + \iint_{Q_3} \varphi^2 \del_t (u_i)_+^{\sigma+2}
\\ &\leq K \abs{\int_{-4}^0 \ddt \int \phi (u_i)_+^{\sigma+2} } + C \paren{\norm{\varphi}_\infty^2 + \norm{\grad\varphi}_\infty^2 }
\\ &\leq K \left[ \int \phi (u_i)_+^{\sigma+2}(0,\cdot) + \int \phi (u_i)_+^{\sigma+2}(-4,\cdot) \right] + C \paren{ \norm{\psi + K \phi}_\infty + \norm{\paren{\grad \sqrt{\psi + K \phi} }^2}_\infty }.
%C \paren{ \norm{\psi + K \phi}_\infty + \norm{\frac{|\grad \psi + K \grad \phi|^2}{\psi + K \phi}}_\infty }.
%\\ &\leq C_\phi K + C \norm{\frac{|\grad \psi + K \grad \phi|^2}{\psi + K \phi}}_\infty }.
\end{align*}
By the chain rule, this last term becomes
\begin{align*}
2 \norm{(\grad \sqrt{\psi + K \phi} )^2}_\infty &= \norm{\frac{|\grad \psi + K \grad \phi|^2}{\psi + K \phi}}_\infty 
\\ &= \sup\paren{ \norm{\frac{|\grad \psi + K \grad \phi|^2}{\psi + K \phi}}_{L^\infty(Q_2)} , \norm{\frac{|\grad \psi + K \grad \phi|^2}{\psi + K \phi}}_{L^\infty(Q_3 \setminus Q_2)} }
\\ &= \sup\paren{ \norm{\frac{|\grad \psi|^2}{\psi + K}}_{L^\infty(Q_2)} , \norm{\frac{|K \grad \phi|^2}{K \phi}}_{L^\infty(Q_3 \setminus Q_2)} }
\\ &\leq \sup\paren{ \frac{1}{\norm{\grad \psi}_\infty}\norm{|\grad \psi|^2}_{\infty}, \frac{K^2}{K} \norm{ \grad \sqrt{\phi} }_\infty^2 }
\\ &\leq C_\phi K.
\end{align*}
In the above calculation, remember that $\phi$ is constant on $Q_2$ and $\psi = 0$ outside $Q_2$, that $\psi + K \geq \norm{\grad \psi}_\infty$ by the definition of $K$, and that $\sqrt{\phi}$ is smooth by assumption.  

We see now that
\[ \chevron{\psi,\del_t (u_i)_+^{\sigma+2}} \leq C(\Lambda,p,n,\phi) \paren{\norm{\psi}_\infty + \norm{\grad \psi}_\infty}\]
and, by duality, $\del_t (u_i)_+^{\sigma+2}$ is bounded in $\M([-4,0];W^{-1,\infty}(B_2))$.  

In order to apply our compactness lemma, we need $(u_i)_+^{\sigma+2}$ to be absolutely continuous in time (i.e. we want $L^1$, not $\M$).  Therefore consider a family of mollifiers $\eta_\delta$ tending to a dirac measure as $\delta \to 0$.  Convolving with respect to time, we obtain smooth-in-time functions.  
\[ \eta_\delta \ast (u_i)_+^{\sigma+2} \in L^p([-4,0]; W^{1,p}(B_2)), \qquad \del_t \left[\eta_\delta \ast (u_i)_+^{\sigma+2} \right] \in L^1([-4,0];W^{-1,\infty}(B_2)) \]
are uniformly bounded independent of $\delta < 1$.  

The Aubin-Lions Lemma indicates that the family $\eta_\delta \ast (u_i)_+^{\sigma+2}$ is compact in $L^1([-4,0]\times B_2)$.  Choose a sequence $\delta_i \to 0$ such that
\[ \norm{(u_i)_+^{\sigma+2} - \eta_{\delta_i} \ast (u_i)_+^{\sigma+2}}_{L^1} \leq \frac{1}{i}.\]
By compactness, the sequence $\eta_{\delta_i} \ast (u_i)_+^{\sigma+2}$ has a subsequential limit $v$, and
\[ \norm{(u_i)_+^{\sigma+2} - v}_1 \leq \norm{(u_i)_+^{\sigma+2} - \eta_{\delta_i} \ast (u_i)_+^{\sigma+2}}_1 + \norm{\eta_{\delta_i} \ast (u_i)_+^{\sigma+2} - v}_1 \to 0.\]  

That is to say, $(u_i)_+^{\sigma+2} \to v$ in $L^1(Q_2)$.  

\step{Showing that the limit engenders a contradiction}

By a measure-theoretic argument, 
\begin{align*}
\tag{*} |\{v \leq 0\} \cap Q_2| &\geq \frac{|Q_2|}{2}, \label{measure_of_v_half} \\
\tag{**} |\{v \geq 1\} \cap \overline{Q}_2| &\geq \delta_0, \textrm{ and} \label{measure_of_v_delta} \\
|\{ 0 < v < 1 \} \cap Q_2| &= 0. 
\end{align*}

The map $f \mapsto \norm{\grad f}_{L^p(Q_2)}$ is lower-semi-continuous on $L^1(Q_2)$, and hence 
\[ \int_{-4}^0 \norm{\grad v}_{L^p(B_2)}^p \,dt < \infty.\]
This implies that for almost every $t \in [-4,0]$, $\norm{\grad v}_p$ is finite; and for such $t$, $v$ must have no spatial jump discontinuities.  In other words, there are three kinds of $t \in [-4,0]$: those at which $v$ is identically 0, those at which $v(t,x) \geq 1$ $\forall x \in B_2$, and the exceptions which have measure zero in $[-4,0]$.  

If we define a new smooth cutoff $\phi$ on $B_2$, and set
\[ H(t) = \norm{\phi^2(\cdot)v(t,\cdot)}_{L^1(B_2)},\]
then for a.e. $t$, either $H(t) = 0$ or $H(t) \geq \norm{\phi^2}_1$.  

On the other hand, we know that $H$ cannot have (certain kinds of) jump discontinuities.  Because $(u_i)_+^{\sigma+2} \to v$ in $L^1(Q_2)$, we know that 
\[ H_i \equiv \norm{\phi^2 (u_i)_+^{\sigma+2}}_1 \longrightarrow H \qquad \textrm{in } L^1([-4,0]).\]  
And by the Energy Inequality~\eqref{energy_3}, with cutoff $\phi$ and $b = \sigma+1$, the derivative of each $H_i$ is bounded uniformly in $i$: notice that $\del_t \phi = 0$ and so for any time interval $[s,t]$ we have
\begin{align*}
H_i(t) - H_i(s) &= \int_s^t \ddt \int \phi^2 (u_i)_+^{\sigma+2}
\\ &\leq C(p,\Lambda,\phi) \iint_s^t \paren{(u_i)_+^{\sigma+1} + (u_i)_+^{\sigma+2} + (u_i)_+^1} \indic{\supp(\phi)}
\\ &\leq [s-t] C(p,\Lambda,\phi).
\end{align*}
Therefore (again by lower-semi-continuity), $\ddt H$ is bounded above.  

This means in particular that if $H(s) = 0$, then $H(t) = 0$ $\forall t \geq s$.  And we know by \eqref{measure_of_v_half} that $v = 0$ on a set of large measure.  In fact, necessarily $H(t) = 0$ $\forall t \in (-2,0]$.  This contradicts \eqref{measure_of_v_delta}, and so the proposition is proven.  

\end{proof}

%****************************************************************************************
%****************************************************************************************
%****************************************************************************************
\section{Transporting improvement forwards in time}\label{sec:flowforward}

Using the propositions proven thus far, one can show, under the appropriate hypotheses, that if a solution to Inequality~\eqref{eq:mainPDEdist} is $\geq -2$ in $Q_3$, then it is in fact $\geq -2 + \varepsilon$ in $[-4,-3]\times B_\varepsilon$.  This is not quite what we set out to prove; we want solutions to become regular after some time elapses, and hence the lower bound must be somewhere in the region $[-1,0]\times B_1$.  

To bridge the gap, we use a barrier function to "flow" the improvement forward in time.  Our solution will still be $\geq -2 + \varepsilon'$ on a ball of radius $\varepsilon'$ at the end of the time interval, and though $\varepsilon'$ becomes smaller as time elapses, it never vanishes entirely.  

This is the first time we use \eqref{eq:mainPDEvisc}.  This inequality is true only in a viscosity sense, so instead of energy methods, we must construct a barrier function which constitutes a subsolution to 
\[ \del_t u + \Lambda |\grad u|^p - \Lambda_0 m^-(D^2 u) = 0. \]

\begin{proposition}\label{th:flowforward}
There exists a constant $0 < K_0 < 1$ depending only on $p$, $\Lambda$, and $n$ such that the following holds:
Let $0 < \lambda \leq K_0$ be a constant and $u$ a viscosity supersolution to Inequality~\eqref{eq:mainPDEvisc} on the interior of $[0,T]\times B_2$ with $T < 4$ and $\Lambda_0 \leq \lambda^2 K_0$.  Suppose that
\[ u \geq -2 \quad \textrm{on} \quad [0,T] \times B_2,\]
\[ u \geq -2 + \lambda^2 \quad \textrm{on} \quad 0 \times B_\lambda.\]
Then
\[ u \geq -2+\frac{\lambda^2}{2} \quad \textrm{on} \quad [0,T] \times B_{\lambda/2}.\]  
\end{proposition}

\begin{proof}
%It will simplify the calculations to work on the set $[0,4]\times B_2$ instead of $[-4,0]\times B_2$.  Since our PDE is translation invariant, this only affects notation.  
%
We define the barrier function
\[ \sigma(t,x) := -2 + \lambda^2 \beta\paren{\frac{|x|}{\lambda}} - \frac{\lambda^2}{8} t,\] 
where $\beta:\R^+ \to \R$ is a smooth function supported on $[0,1]$ and identically 1 on $[0,1/2]$.  

If we can show that $\sigma$ is a subsolution to \eqref{eq:mainPDEvisc}, and that it is less than $u$ on the parabolic boundary $0 \times B_2 \cup [0,T]\times \del B_2$, then the standard theory of comparison principles tells us that $u \geq \sigma$ on the whole interior of $[0,T]\times B_2$.  See \cite{cran} for the elliptic version of the comparison principle, and \cite{j} for a treatment more specific to the parabolic case.  

In particular, for $(t,x) \in [0,T] \times B_{\lambda/2}$ we have
\[\sigma(t,x) = -2 + \lambda^2 (1-t/8) \geq -2 + \lambda^2 (1-T/8) \geq -2 + \lambda^2/2. \]  
Thus showing $u \geq \sigma$ will prove the proposition.  

\step{Barrier is below $u$ on the boundary}

At $t = 0$, 
\[ \sigma(0,x) \leq -2 + \lambda^2 \leq u \qquad \forall \, x \in B_\lambda, \]
\[ \sigma(0,x) \leq -2  \leq u \qquad \forall \, x \in B_2 \setminus B_\lambda; \]
and on the spatial boundary $|x|=2$, 
\[ \sigma(t,x) = -2 - \frac{\lambda^2}{8} t \leq -2 \leq u \qquad \forall \, t \in [0,T].\]
Thus on the parabolic boundary of $[0,T] \times B_2$, we have $\sigma \leq u$.  

\step{Barrier is a subsolution}

By construction 
\[ \del_t \sigma(t,x) = -\lambda^2/8 \]  
and
\[ |\grad \sigma|(t,x) = \lambda \beta'\paren{\frac{|x|}{\lambda}}. \] 
To compute $D^2 \sigma$, notice that $\sigma$ is radially symmetric in space, and so it suffices to compute the Hessian at the point $x = (|x|,0,\ldots,0)$.  At this point, one can compute directly that
\begin{align*} 
\del_{11} \sigma(t,x) &= \frac{d^2}{dh^2} \bigg\rvert_{h=0}  \lambda^2 \beta\paren{\frac{|x|+h}{\lambda}} 
\\ &= \beta''\paren{\frac{|x|}{\lambda}}
\end{align*}
and for $i \neq 0$
\begin{align*} 
\del_{ii} \sigma(t,x) &= \frac{d^2}{dh^2} \bigg\rvert_{h=0}  \lambda^2 \beta\paren{\frac{\sqrt{|x|^2 + h^2}}{\lambda}} 
\\ &= \frac{\lambda}{|x|} \beta'\paren{\frac{|x|}{\lambda}}.
\end{align*}
For any $i \neq j$, assume without loss of generality that $i \neq 1$.  Then $[\del_i \sigma](x) = 0$ for any $x$ in the hyperplane $x_i = 0$, by radial symmetry.  Therefore $\del_j [\del_i \sigma] = 0$ at $(|x|,0,\ldots,0)$.  

We conclude that the matrix $D^2 \sigma (t,x)$ is a diagonal matrix with eigenvalues 
\[ \frac{\lambda}{|x|} \beta'\paren{\frac{|x|}{\lambda}} \quad \textrm{and} \quad \beta''\paren{\frac{|x|}{\lambda}}, \]
and by symmetry it should have the same eigenvalues at generic $x$.  

Therefore, to see if $\sigma$ is a subsolution, calculate
\begin{align*}
\del_t \sigma + \Lambda |\grad \sigma|^p - \Lambda_0 m^-(D^2 \sigma) &= -\frac{\lambda^2}{8} + \Lambda \lambda^p (\beta')^p - \Lambda_0 \min\left( \beta'', \frac{\lambda}{|x|} \beta', 0 \right)
\\ &\leq \frac{-\lambda^2}{8} + \Lambda \lambda^p \norm{\beta'}_\infty^p + \Lambda_0 \norm{\beta''}_\infty + \Lambda_0 \frac{\lambda}{1/2} \norm{\beta'}_\infty
\\ &\leq \frac{-\lambda^2}{8} + \Lambda \lambda^p \norm{\beta'}_\infty^p + \lambda^2 K_0 \norm{\beta''}_\infty + 2 \lambda^3 K_0 \norm{\beta'}_\infty
\\ &= \lambda^2 \paren{\Lambda \lambda^{p-2} \norm{\beta'}_\infty^p + K_0 \norm{\beta''}_\infty + 2 \lambda K_0 \norm{\beta'}_\infty - \frac{1}{8}}
\\ &\leq \lambda^2 \paren{\Lambda K_0^{p-2} \norm{\beta'}_\infty^p + K_0 \norm{\beta''}_\infty + K_0^2 \frac{n-1}{1/2} \norm{\beta'}_\infty - \frac{1}{8}}.
\end{align*}

This last quantity is negative provided $K_0$ sufficiently small, depending on $\Lambda$, $p$, the dimension, and the specific choice of $\beta$.  

\end{proof}

\section{Proof of the main theorem} \label{sec:main}

Having completed the core of the proof, we now come to the final section.  The pieces are all present, and we need only put them together.  This section contains three lemmas before the proof.  The first two (Lemmas \ref{th:scaling1} and \ref{th:scaling2}) tell us which scalings constitute symmetries of our PDE.  Lemma \ref{th:oscillation}, the Oscillation Lemma, applies Propositions~\ref{th:DG1} and \ref{th:DG2} iteratively in order to control the oscillation of solutions to our PDE.  Finally the proof of the Main Theorem will show how the Oscillation Lemma is equivalent to interior H\"{o}lder continuity.  

The proof of the Oscillation Lemma is slightly non-standard.  The rest is technical, with no new ideas.  

\begin{lemma} \label{th:scaling1}
 If $u$ satisfies the two equations \eqref{eq:mainPDEdist} and \eqref{eq:mainPDEvisc} on a cylinder $[T_0,0]\times\Omega$, and $\alpha,\beta > 0$ are any two real numbers satisfying 
\[ \beta \leq \alpha\n \quad \beta \leq \alpha^{-\frac{p-1}{p-2}}, \quad \beta \leq \alpha^{-\frac{p(m-1)+1}{p(m-1) - n} } , \]
then the modified function 
\[ v(t,x) := \alpha u(\alpha^{p-1}\beta^p t, \beta x)\] 
satisfies the equations
\begin{align*}
\del_t v + \Lambda\n \abs{\grad v}^p - \div(A' \grad v) &\leq f' \\
\del_t v + \Lambda \abs{\grad v}^p - \Lambda_0' m^-(D^2 v)  &\geq 0
\end{align*}
on $\left[ \frac{T_0}{\alpha^{p-1}\beta^p},0\right] \times \frac{1}{\beta} \Omega$, with $\Lambda_0' = \alpha^{p-1}\beta^{p-2} \Lambda_0 \leq \Lambda_0$, $\norm{A'}_\infty \leq \norm{A}_\infty$ and $\norm{f'}_m \leq \norm{f}_m$.
%If $\varepsilon, \norm{f}_m \leq \Lambda$ then $\varepsilon', \norm{f'}_m \leq \Lambda$ as well.  
\end{lemma}

\begin{proof}
One must take
\begin{align*}
f'(t,x) &:= \alpha^p \beta^p f(\alpha^{p-1}\beta^p t, \beta x), \\
A'(t,x) &:= \alpha^{p-1}\beta^{p-2} A(\alpha^{p-1}\beta^p t, \beta x).
\end{align*}

Applying our differential operator to $v$, we obtain 
\begin{align*}
\del_t v + \Lambda\n \abs{\grad v}^p - \div(A' \grad v) &= (\alpha\beta)^p \del_t u + (\alpha\beta)^p \Lambda\n \abs{\grad u}^p - (\alpha \beta)^p  \div(A \grad u)
\\ &= (\alpha\beta)^p \left[ \del_t u + \Lambda\n \abs{\grad u}^p - \div (A \grad u) \right]
\\ &\leq f'
\end{align*}

For the other inequality, similarly, 
\begin{align*}
\del_t v + \Lambda \abs{\grad v}^p - \Lambda_0 m^-(D^2 v) &= (\alpha\beta)^p \del_t u + (\alpha\beta)^p \Lambda \abs{\grad u}^p - \alpha \beta^2 \Lambda_0 m^-(D^2 u)
\\ &= (\alpha\beta)^p \left[ \del_t u + \Lambda \abs{\grad u}^p - \Lambda m^-(D^2 u) \right]
\\ &\geq 0.
\end{align*}

%\[ f'(t,x) := \alpha^p \beta^p f(\alpha^{p-1}\beta^p t, \beta x). \]

That $\Lambda_0' \leq \Lambda_0$ and $\norm{A'}_\infty \leq \norm{A}_\infty$ follows immediately from our assumptions on $\alpha$, $\beta$.  For $\norm{f'}_m$, we notice that $p - \frac{p+n}{m}$ is necessarily positive, and calculate
\begin{align*}
\norm{\alpha^p \beta^p f(\alpha^{p-1}\beta^p t, \beta x)}_m &= \alpha^p \beta^p (\alpha^{p-1}\beta^p \beta^n )^{-1/m} \norm{f}_m 
\\ &= \alpha^{p - \frac{p-1}{m}} \beta^{p - \frac{p+n}{m}} \norm{f}_m
\\ &\leq \alpha^{p - \frac{p-1}{m}} \paren{\alpha^{-\frac{p(m-1)+1}{p(m-1) - n} }}^{p - \frac{p+n}{m}} \norm{f}_m = \norm{f}_m.
\end{align*}
%This calculation used that $p - \frac{p+n}{m}$ is positive, which follows from $m > 1 + \frac{n}{p}$.  

\end{proof}

\begin{lemma} \label{th:scaling2}
If $u$ satisfies Inequality~\eqref{eq:mainPDEdist} on a cylinder $[T_0,0]\times\Omega$, there exist constants $e_1 \in (2,p)$ and $e_2 < 0$ dependent on $n$, $m$, $p$ such that, for any two real numbers $0 < \beta \leq 1$ and $1\leq \alpha \leq \beta^{e_2}$,  %$\min\paren{\beta^{-\frac{p-e}{p-1}}, \beta^{-p\paren{1-\frac{1 + n/p}{m}}} }$,
%\[0 < \beta \leq 1$, $1 \leq \alpha \leq \beta^\frac{2-p}{p-1}$ and $\alpha \leq \beta^{2 - \frac{2+n}{m},\] 
the modified function 
\[ v(t,x) := \alpha u(\beta^{e_1} t, \beta x)\]
also satisfies Inequality~\eqref{eq:mainPDEdist} on $\left[ T_0,0\right] \times \Omega$ with parameters $\norm{f'}_m \leq \norm{f}_m$, $\norm{A'}_\infty \leq \norm{A}_\infty$ and the same $\Lambda$.  
\end{lemma}

\begin{proof}
Since $\frac{n}{m-1} < p$ and $p > 2$, we can choose a constant $e_1 \in (\frac{n}{m-1},p)$ such that $e_1 > 2$.  Let
\[ e_2 := \max\paren{-\frac{p-e_1}{p-1}, \frac{n}{m} - e_1\frac{m-1}{m} } \]
so that
\[ \alpha^{p-1} \beta^{p-e_1} = \paren{\alpha \paren{\frac{1}{\beta}}^{-\frac{p-e_1}{p-1}}}^{p-1} \leq \paren{\alpha \paren{\frac{1}{\beta}}^{e_2}}^{p-1} \leq 1\]
 and $\alpha \beta^{e_1\frac{m-1}{m}-\frac{n}{m}} \leq 1$.  

Define 
\begin{align*}
A'(t,x) &:= \beta^{e_1-2} A(\beta^{e_1} t, \beta x), \\
f'(t,x) &:= \alpha \beta^{e_1} f(\beta^{e_1} t, \beta x).
\end{align*}
Applying our differential operator to $v$, we obtain 
\begin{align*}
\del_t v + \Lambda\n \abs{\grad v}^p + \div(A' \grad v) &= \alpha \beta^{e_1} \del_t u + (\alpha\beta)^p \Lambda\n \abs{\grad u}^p + \alpha \beta^{e_1} \div(A \grad u)
\\ &= \alpha\beta^{e_1} \left[ \del_t u + \paren{\alpha^{p-1}\beta^{p-e_1}} \Lambda\n \abs{\grad u}^p + \div(A \grad u) \right]
\\ &\leq \alpha\beta^{e_1} \left[ \del_t u + \Lambda\n \abs{\grad u}^p + \div(A \grad u) \right]
\\ &\leq \alpha\beta^{e_1} f = f'.
\end{align*}

That $\norm{A'}_\infty \leq \norm{A}_\infty$ follows immediately from our assumption that $e_1 > 2$.  It remains to calculate the norm of $f'$:
\begin{align*}
\norm{f'}_m &= \alpha \beta^{e_1} (\beta^{e_1} \beta^n)^{-1/m} \norm{f}_m 
\\ &= \alpha \beta^{e_1(1-\frac{1}{m}) - \frac{n}{m}} \norm{f}_m
\\ &\leq \norm{f}_m.
\end{align*}

A priori, $v$ will satisfy this inequality on $\left[ \frac{T_0}{\beta^{e_1}},0\right] \times \frac{1}{\beta} \Omega$.  Since we assume $\beta \leq 1$, this in particular means it is satisfied on $[T_0,0] \times \Omega$.

\end{proof}

At last we can prove the Oscillation Lemma.  The oscillation of a function is the distance between its supremum and its infimum, and for solutions of \eqref{eq:mainPDEdist} and \eqref{eq:mainPDEvisc}, if the oscillation is finite on a region it will be strictly less on a strictly smaller region.  

\begin{lemma}[Oscillation Lemma] \label{th:oscillation}
There exist constants $\lambda^*>0$, $r^*>0$, $T^* < 0$ depending on $\Lambda$, $p$, $n$, $\mu_0$ (from Proposition~\ref{th:DG2}), $\delta_0$ (from Proposition~\ref{th:DG1}), $K_0$ (from Proposition~\ref{th:flowforward}), and $e_1$, $e_2$ (from Lemma~\ref{th:scaling2}) such that, for any solution $u$ to Inequalities~\eqref{eq:mainPDEdist} and \eqref{eq:mainPDEvisc} on $Q_3$, with $\Lambda_0 < (\lambda^*)^2 K_0$, we have the following implication:
If
\[ |u| \leq 2 \qquad \forall\, (t,x) \in Q_3, \]
then either
\[ \sup_{[T^*,0]\times B_{r^*}(0)} u \leq 2 - \frac{(\lambda^*)^2}{2}\]
or
\[ \inf_{[T^*,0]\times B_{r^*}(0)} u \geq -2 + \frac{(\lambda^*)^2}{2}.\]  

\end{lemma}

The idea of the proof is to apply De Giorgi's First Lemma to some truncation of $u$.  Remember that De Giorgi's First Lemma says that if the measure of $\{u_+>0\}$ is sufficiently small, then $u_+$ is $L^\infty$-bounded on some smaller domain.  This $L^\infty$ bound is precisely what we wish to prove.  We attempt to apply the lemma to each of $(u-C_k)_+$ for $C_k$ an increasing series of constants.  Obviously the measure shrinks as $C_k$ increases; De Giorgi's Second Lemma allows us to quantify the decrease in measure, and find a precise $k$ for which De Giorgi's First Lemma applies.  

\begin{proof}
Let $k_0$ be the smallest integer greater than $|Q_2|/\mu_0$,
%Set
%\[ k_0 = \operatorname{ceiling}\paren{\frac{|Q_2|}{\mu_0}}, \]
where $\mu_0$ is the constant in Proposition~\ref{th:DG2}, and define 
\[ Q_\textrm{small} := [-4 \cdot 2^{k_0 e_1 / e_2},0] \times B_{2 \cdot 2^{k_0 / e_2}}. \]

There are two cases to consider: either we will upper-bound the supremum or we will lower-bound the infimum of $u$ in the region $[T^\ast,0] \times B_{r^\ast}(0)$.  If
\[ |\{u \leq 0\} \cap Q_\textrm{small}| \geq \frac{|Q_\textrm{small}|}{2}, \]
we are in the former case, so we call $u$ ``mostly negative'' and define 
\[ v(t,x) := u(2^{k_0 e_1/e_2} t , 2^{k_0 / e_2} x). \]
Otherwise, we are in the latter case, so we call $u$ ``mostly positive'' and define
\[ v(t,x) := - u(2^{k_0 e_1 / e_2} (-4 - t) , 2^{k_0 / e_2} x). \]
In either case,
\[ |\{v \leq 0\} \cap Q_2| \geq \frac{Q_2}{2}. \]

For integers $k \in [0,k_0]$ consider the functions 
\[ v_k = 2^k (v - 2) + 2.\]  
Notice that for all $k \leq k_0$, $v_k \leq 2$ on $Q_3$.  By Lemma~\ref{th:scaling2} with $\alpha = 2^k$ and $\beta = 2^{k_0 / e_2}$ and domain $Q_3$, combined with the fact that Inequality~\eqref{eq:mainPDEdist} is preserved by translations, addition of constants, and the transformation $f(t,x) \mapsto -f(-t,x)$, each $v_k$ satisfies Inequality~\eqref{eq:mainPDEdist} on $Q_3$.  

We claim that $|\{v_{k_0} \geq 1\} \cap \overline{Q}_2|\leq \delta_0$.  If this were not the case, then in fact
\[ |\{v_k \geq 1\} \cap \overline{Q}_2| > \delta_0,\]
for all $k \leq k_0$, because the quantity is non-increasing as $k$ increases.  Similarly, 
\[ |\{v_k \leq 0\} \cap Q_2| \geq \frac{|Q_2|}{2}\]
for all $k \leq k_0$, because the same holds for $v_0$ and the quantity is non-decreasing.  

This is enough for us to apply De Giorgi's Second Lemma to each $v_k$.  By construction, the Lemma tells us that
\[ |\{v_{k+1} \geq 0\} \cap Q_2| \leq |\{v_k \geq 0\} \cap Q_2| - \mu_0.\]  

This cannot possibly be true for all $k$ between 0 and $k_0$, since $k_0 \mu_0 > |Q_2|$.  This is a contradiction.  

Therefore $|\{v_{k_0} \geq 1\} \cap \overline{Q}_2|\leq \delta_0$.  We can apply De Giorgi's First Lemma to $v_{k_0}-1$, and learn that $v_{k_0} \leq 3/2$ on $Q_1$.  In terms of $v$,  
\[ v(t,x) \leq 2 - 2^{-k_0-1} \qquad \forall \, (t,x) \in Q_1. \]

In the case that $u$ is mostly negative, this means
\[ u(t,x) \leq 2 - 2^{-k_0-1} \qquad \forall \, (t,x) \in [T,0] \times B_{r}(0), \qquad T= -2^{k_0 e_1 / e_2}, \,\, r = 2^{k_0 / e_2} \]
and the proof is complete.  So consider the case where $u$ is mostly positive.  We've shown that
\[ u \geq -2 + 2^{-k_0-1} \qquad \forall (t,x) \in [-4 \cdot 2^{k_0 e_1 / e_2}, -3 \cdot 2^{k_0 e_1 / e_2}]\times B_r.\]
The problem here is the time interval; we want a lower bound on the infimum of $u$ in a parabolic neighborhood of $(0,0)$.  Define 
\[ \lambda^* = \min(K_0, \sqrt{2^{-k_0-1}}). \]
Proposition~\ref{th:flowforward} applied to the lower-semicontinuous envelope of $u$ tells us that, since we assumed $\Lambda_0 \leq (\lambda^*)^2 K_0$,  
\[ u \geq -2 + \frac{(\lambda^*)^2}{2} \qquad \textrm{on } [4T,0] \times B_{\lambda^*/2}.\] 

Letting $T^* = T$, $r^* = \min(r, \lambda^*/2)$, we see that either
\[ \sup_{[T^*,0]\times B_{r^*}(0)} u \leq 2 - \frac{(\lambda^*)^2}{2}\]
or
\[ \inf_{[T^*,0]\times B_{r^*}(0)} u \geq -2 + \frac{(\lambda^*)^2}{2}.\]

\end{proof}

Finally, we are ready to prove Theorem~\ref{th:main}.  

\begin{proof}
Instead of proving continuity directly for $u$, it is preferable to consider
\[ \bar{u} \equiv u + \Lambda t,\]
which satisfies the Inequalities~\eqref{eq:mainPDEdist} and \eqref{eq:mainPDEvisc}.  Clearly $\bar{u}$ and $u$ will have the same H\"{o}lder exponent.  

Since $\bar{\Omega}$ is compact, there is a radius $\rho$ such that $B_{\rho}(x) \subseteq \Omega$ for each $x \in \bar{\Omega}$.  

%To control the $L^\infty$ norm of $\bar{u}$, note that for any point $(\tilde{t},\tilde{x}) \in (s,T)\times \bar{\Omega}$, if $c = \min(\rho/2, \sqrt{s/4})$ then the function
%\[ \min\left(\frac{\sqrt[\ell]{\delta_0 c}}{\norm{\bar{u}}_{\ell}^{\ell}}, c^{\frac{p-2}{1-p}}\right) \bar{u}(c^2 (t - \bar{t}), c (x - \bar{x})) \]
%satisfies all the requirements of Lemma~\ref{th:scaling2} and Proposition~\ref{th:DG1}.  Therefore $\bar{u}$ must be bounded above on $(s/2,T)\times \{x: |x-\bar{\Omega}|<\rho\}$, and bounded below on $(0,T-s/2)\times \{x: |x-\bar{\Omega}|<\rho\}$.  To extend the lower bound all the way to $T$, notice that any constant function is a subsolution to Inequality~\eqref{eq:mainPDEvisc} so the infimum over any region in space cannot decrease over time.  
%Therefore $\norm{\bar{u}}_{L^\infty((s/2,T)\times \{x: |x-\bar{\Omega}|<\rho\})}$ is finite, and depends only on $p$, $n$, $\Lambda$, $\rho$, and $s$.  

Consider any two points $(t_0,x_0), (t_1,x_1) \in (s,T)\times\bar{\Omega}$, and assume wlog that $t_0 \geq t_1$.  If these points are far away, then we can estimate the H\"{o}lder norm in a very rough way, using the $L^\infty$ norm of $\bar{u}$.  If the points are very close together, then we must use the Oscillation Lemma.  

We want to rescale the function $\bar{u}$ to obtain $w$ centered at $(t_0,x_0)$ but solving the PDE on $Q_3$, with $\norm{w}_\infty \leq 2$, and with $\Lambda_0 \leq (\lambda^*)^2 K_0$.  To that end, choose $\alpha_w,\beta_w$ small enough that 
\[ \alpha_w \leq \frac{2}{\norm{\bar{u}}_{L^\infty([T,0]\times \Omega)}}, \qquad 3\beta_w \leq \rho, \qquad 4 \alpha_w^{p-1}\beta_w^p \leq s, \qquad \alpha_w^{p-1} \beta_w^{p-2} \Lambda_0 \leq (\lambda^*)^2 K_0, \]
and
\[\alpha_w\beta_w \leq 1, \qquad \alpha_w^{p-1}\beta_w^{p-2} \leq 1, \qquad \alpha^{p(m-1)+1} \beta^{p(m-1)-n} \leq 1.\]
Note that $\alpha_w$ and $\beta_w$ depend on $\norm{u}_{L^\infty}$.  

Lemma~\ref{th:scaling1} tells us that
\[ w(t,x) := \alpha_w \bar{u}\paren{t_0 + \alpha_w^{p-1}\beta_w^p t, x_0 + \beta_w x} \]
is a solution to Inequalities~\eqref{eq:mainPDEdist} and \eqref{eq:mainPDEvisc} on $Q_3$, with $\Lambda_0 \leq (\lambda^\ast)^2 K_0$.  By construction $|w|\leq 2$ on $Q_3$.  %Therefore $w$ satisfies the hypotheses of Lemma~\ref{th:oscillation}.  

Now that $w$ is formatted correctly, the plan is to apply Lemma~\ref{th:oscillation} iteratively, showing that the oscillation of $w$ decreases as the distance to $(0,0)$ decreases.  

Set
\[ \alpha_1 = \frac{4}{4-(\lambda^*)^2/2},\]
and take $\beta_1$ small enough that $3 \beta_1 \leq r^*$, and $4\alpha_1^{p-1}\beta_1^p \leq -T^*$, and small enough to satisfy the hypotheses of Lemma~\ref{th:scaling1}.  Define $w_0 = w$ and iteratively define
\[ w_{k+1}(t,x) := \alpha_1 \left[ w_k(\alpha_1^{p-1} \beta_1^p t, \beta_1 x) \pm \frac{(\lambda^*)^2}{4} \right], \]
with $\pm$ chosen as whichever sign minimizes $\norm{w_{k+1}}_{L^\infty(Q_3)}$.  By induction, $|w_k| \leq 2$ on $Q_3$ and $w_k$ solves Inequalities~\eqref{eq:mainPDEdist} and \eqref{eq:mainPDEvisc} on $Q_3$ with $\Lambda_0 \leq (\lambda^\ast)^2 K_0$, and hence satisfies the hypotheses of Lemma~\ref{th:oscillation}.  

Therefore, for all $k \geq 0$, we find that for $Q_k = [-(\alpha_1^{p-1}\beta_1^p)^k,0]\times B_{\beta_1^k}$,
\[ \sup_{Q_k} w(t,x) \, - \, \inf_{Q_k} w(t,x) \leq \frac{1}{\alpha_1^{k-1}} \paren{4-\frac{(\lambda^*)^2}{2}}. \]

Remember that we are trying to bound the H\"{o}lder norm, the quantity
\[ (*) = \frac{|\bar{u}(t_1,x_1) - \bar{u}(t_0,x_0)|}{|(t_0-t_1)^2 + |x_0-x_1|^2|^{\gamma/2}}. \]

If $\sqrt{(t_0-t_1)^2 + |x_0 - x_1|^2} \geq \alpha_w^{p-1} \beta_w^p$, then we can bound
\[ (*) \leq \frac{2 \norm{\bar{u}}_\infty}{( \alpha_w^{p-1}\beta_w^p )^\gamma}. \]

Otherwise, we can use the control on the oscillation of $w$.  Specifically, if
\[ \sqrt{(t_0-t_1)^2 + |x_0-x_1|^2} \leq \alpha_w^{p-1} \beta_w^p (\alpha_1^{p-1}\beta_1^p)^k \]
%\[ t_0-t_1 \leq \alpha_w^{p-1} \beta_w^p (\alpha_1^{p-1}\beta_1^p)^k \qquad \textrm{and} \qquad |x_1-x_0| \leq \beta_w \beta_1^k, \] 
for any integer $k \geq 0$, then, because $\alpha_w \beta_w \leq 1$ and $\alpha_1 \beta_1 \leq 1$, 
\[ \paren{\frac{t_1-t_0}{\alpha_w^{p-1} \beta_w^p}, \frac{x_1-x_0}{\beta_w}} \in Q_k.\]
%\[ \sqrt{(t_0-t_1)^2 + |x_0-x_1|^2} \leq \frac{(\alpha_1\beta_1)^{(p-1)k} \beta_1^k}{2} \]
Therefore
\[ \abs{ w\paren{\frac{t_1-t_0}{\alpha_w^{p-1} \beta_w^p}, \frac{x_1-x_0}{\beta_w}} - w(0,0) } = \alpha_w \abs{\bar{u}(t_1,x_1) - \bar{u}(t_0,x_0)} \leq \frac{4-\frac{(\lambda^*)^2}{2}}{\alpha_1^{k-1}}. \]
This relationship implies that 
\begin{align*}
|\bar{u}(t_1,x_1) - \bar{u}(t_0,x_0)| &\leq \paren{4-\frac{(\lambda^*)^2}{2}} \bigg/ \paren{ \alpha_w \alpha_1^{ \frac{\log\paren{\sqrt{(t_0-t_1)^2 + |x_0-x_1|^2}\big/(\alpha_w^{p-1} \beta_w^p)}}{\log(\alpha_1^{p-1} \beta_1^p)} - 2} }
\\ &\leq \paren{4-\frac{(\lambda^*)^2}{2}} \frac{\alpha_1^2}{\alpha_w} \alpha_1^\frac{\log(\alpha_w^{p-1} \beta_w^p)}{\log(\alpha_1^{p-1} \beta_1^p)} \sqrt{(t_0-t_1)^2 + |x_0-x_1|^2}^{\paren{\frac{-\log(\alpha_1)}{\log(\alpha_1^{p-1} \beta_1^p)}}}. 
\end{align*}

Hence if 
\[ \gamma = \frac{-\log(\alpha_1)}{\log(\alpha_1^{p-1} \beta_1^p)},\]
then
\[ (*) \leq \paren{4-\frac{(\lambda^*)^2}{2}} \frac{\alpha_1}{\alpha_w} \alpha_1^\frac{\log(\alpha_w^{p-1} \beta_w^p)}{\log(\alpha_1^{p-1} \beta_1^p)}. \]

Note that the bound depends non-linearly on $\alpha_w$ and $\beta_w$, and hence on $\norm{u}_{\infty}$, but $\gamma$ depends only on $n$, $p$, $m$, $\Lambda$, and $\Lambda_0$.  

This completes the proof.

\end{proof}

\bibliography{biblio.bib}
\bibliographystyle{plain}
\end{document}